\documentclass[11pt]{article}
\usepackage{amsfonts}
\usepackage{epsfig,graphicx,epsf,wrapfig}
\usepackage{amsmath,array}
\usepackage{graphicx,epstopdf,epsfig,epic,bm}
\usepackage{setspace}
\usepackage{multirow}
\usepackage{geometry}
\usepackage{float}
\usepackage{fancyhdr}
\usepackage{booktabs,threeparttable}
\usepackage{mathrsfs}
\usepackage{amssymb}
\usepackage{ntheorem}
\newtheorem{theorem}{Theorem}[section]
\newtheorem{lemma}[theorem]{Lemma}

\theoremstyle{definition}

\theoremstyle{remark}
\newtheorem*{proof}{proof}
\newtheorem{remark}[theorem]{Remark}
\usepackage{hyperref}
\hypersetup{hypertex=true,
            colorlinks=true,
            linkcolor=blue,
            anchorcolor=blue,
            citecolor=red}
\numberwithin{equation}{section}
\usepackage{algorithm}
\usepackage{algorithmic}
\usepackage{cases}
\usepackage{lineno}
\begin{document}
\title{\bf A new three-operator splitting method for the monotone inclusion problem\footnote{This work was supported by the National Natural Science Foundation of China under Grant
No. 12471290, and the Postgraduate Research \& Practice Innovation Program of Jiangsu Province
under Grant No. KYCX25\_1928.}}
\author{\normalsize Maoran Wang\thanks{School of Mathematical Sciences, Nanjing Normal University, Nanjing 210023, P.R. China. \tt{\color{blue}230901010@njnu.edu.cn}},\; Zijun Xia\thanks{School of Mathematical Sciences, Nanjing Normal University, Nanjing 210023, P.R. China. \tt{\color{blue}250901016@njnu.edu.cn}},\; Xingju Cai\thanks{School of Mathematical Sciences, Nanjing Normal University, Key Laboratory of NSLSCS (NNU), Ministry of Education, Nanjing 210023, P.R. China. \tt{\color{blue} caixingju@njnu.edu.cn}}
}
 \date{}
  \maketitle

 \begin{abstract}
This paper studies a class of monotone inclusion problems in a real Hilbert space involving the sum of three operators, where two are maximal monotone and the third is cocoercive. The Davis--Yin three-operator splitting method extends the two-operator splitting methods---namely the forward-backward method and the Douglas--Rachford method---to the three-operator setting. In addition, two other common splitting methods for two-operator problems are the reflected forward-backward and forward-reflected-backward methods. While several three-operator extensions exist for each of these methods individually, a unified framework that generalizes both remains absent. This raises the question: can they be extended to the three-operator case within a single algorithm? To address this, we propose a new splitting algorithm that unifies the Douglas--Rachford, reflected forward-backward, and forward-reflected-backward methods as special cases. We prove its weak convergence and establish its sublinear convergence rate for convex optimization problems under appropriate stepsize conditions. Finally, we present numerical experiments to validate the theoretical properties and demonstrate the effectiveness of the proposed method.

 \medskip

\end{abstract}
\section{Introduction}
Let $\mathcal{H}$ be a real Hilbert space equipped with the inner product $\langle \cdot, \cdot \rangle$ and the induced norm $\|\cdot\|$. Let $A, C : \mathcal{H} \rightarrow 2^{\mathcal{H}}$ be maximal monotone operators, and let $B : \mathcal{H} \rightarrow \mathcal{H}$ be a $\beta$-cocoercive operator for some $\beta > 0$.  In this paper, we consider the following monotone inclusion problem:
\begin{equation}\label{cp}
\text{Find } x \in \mathcal{H} \text{ such that } 0 \in Ax + Bx + Cx.
\end{equation}
Throughout this paper, we assume that problem \eqref{cp} has at least one solution. To solve this problem numerically, Davis and Yin introduced the three-operator splitting method (TSM) \cite{three}. The iteration scheme of TSM is defined as follows:
\begin{equation}\label{0}
\begin{cases}
x^{k+1}=J_{\gamma C}(z^k),\vspace{1ex}\\
y^{k+1}=J_{\gamma A}(2x^{k+1}-z^k-\gamma Bx^{k+1}),\vspace{1ex}\\
z^{k+1}=z^k+y^{k+1}-x^{k+1}.
\end{cases}
\end{equation}
Equivalently, the iteration \eqref{0} may be reformulated as a Picard iteration:
\begin{equation*}\label{Davis-Yin}
z^{k+1} = Tz^k:=\left[Id - J_{\gamma C} + J_{\gamma A} \circ \left(2J_{\gamma C} - Id - \gamma B \circ J_{\gamma C}\right)\right]z^k.
\end{equation*}
In \cite{three}, they established that for stepsizes \( \gamma \in (0, 2\beta) \), the operator \( T \) is \( \frac{2\beta}{4\beta - \gamma} \)-averaged, and that \( J_{\gamma C} \left( {\rm Fix}(T) \right) = {\rm zer}(A + B + C) \). By leveraging the Krasnosel'ski\u{\i}–Mann { {iteration}}, they demonstrated the global convergence of the algorithm and derived a sublinear convergence rate. Moreover, several classical operator splitting algorithms arise as special cases of TSM. For instance, the proximal point method ($A = B = 0$) \cite{ppa}, the forward-backward splitting method ($A$ or $C = 0$) \cite{1opt,Beck2009,2017pjo}, the Douglas–Rachford splitting method ($B = 0$) \cite{DR56,Eckstein1992,Lions1979,DR11}, and the forward-Douglas--Rachford splitting method (where $C$ is the normal cone operator of a closed subspace) \cite{2015Forward}. For a detailed discussion of improvements to the relaxed TSM, the reader is referred to \cite{2022DY}.

As a commonly used method in the existing literature for solving the three-operator problem~\eqref{cp}, the three-operator splitting method~\eqref{0} has garnered considerable attention. Zhang and Chen~\cite{pdy} introduced a parameterized variant of TSM, extending the parameterized Douglas–Rachford splitting method~\cite{pdr}, to compute the minimum-norm solution of problem~\eqref{cp}. To improve computational efficiency, Pedregosa and Gidel~\cite{ady} developed an adaptive TSM based on the original algorithm, incorporating line search techniques for operator \(B\).Similar adaptive strategies have also been explored in the context of variational inequality problems, as reported in \cite{peng2026extrapolated,peng2023accelerated,ur2025approximate}. Inertial variants of TSM have also been extensively studied in~\cite{dong2022inertial} and~\cite{idy} (motivated by \cite{BOT}). Recently, the linear convergence rates of TSM have been analyzed in \cite{LDY} using scaled relative graphs. The preconditioned and stochastic variants of TSM have been further explored in~\cite{pddy,pd3o}. For a comprehensive overview of operator splitting algorithms, we refer the reader to~\cite{combettes2024geometry,re}.

It is worth noting that TSM~\eqref{0} reduces to the forward-backward splitting method when either \(A = 0\) or \(C = 0\), and to the Douglas--Rachford splitting method when \(B = 0\). In addition to these, two other common splitting methods for two-operator problems are the forward-reflected-backward (FRB)~\cite{frb} and reflected-forward-backward (RFB)~\cite{rfb} splitting methods. Recently, several three-operator splitting methods have been developed to extend these two methods to the three-operator setting. For instance, the backward-forward-reflected-backward splitting method~\cite{rieger2020backward} and the forward-reflected-Douglas-Rachford splitting method~\cite{ryu2020finding} generalize the FRB method, while the backward-reflected-forward-backward splitting method~\cite{rieger2020backward} reduces to the RFB method. However, each of these extensions is tailored to a specific two-operator method and fails to simultaneously generalize both FRB and RFB to the three-operator setting within a single framework. This naturally raises the following question: does there exist a three-operator splitting method that simultaneously unifies the Douglas--Rachford, FRB, and RFB methods as special cases? In this paper, we address this question by introducing a new three-operator splitting method, which consists of two subproblems: the first involves a forward-backward step, while the second involves a backward step. The proposed algorithm encompasses the Douglas--Rachford, reflected forward-backward, and forward-reflected-backward splitting methods as special cases. Specifically, by setting appropriate operators to zero, each of these classical methods can be recovered within a unified framework.

The structure of this paper is organized as follows. Section~\ref{s2} reviews fundamental definitions and presents preliminary results that are essential for the subsequent analysis. In Section~\ref{s3}, we introduce the main contributions of the paper, namely, the proposed three-operator splitting method, along with its convergence analysis and sublinear convergence rate for convex optimization problems. Section~\ref{s4} reports a numerical experiment that illustrates and compares the performance of the proposed method and the Davis--Yin three-operator splitting method. Finally, Section~\ref{s5} provides concluding remarks and outlines potential directions for future research.
\section{Notations and  preliminary results}\label{s2}
Throughout this paper, $\mathcal{H}$ is a real Hilbert space with the inner product $\langle \cdot,\cdot\rangle$ and associated norm $\|\cdot\|$. The symbols ``$\to$'' and ``$\rightharpoonup$'' denote strong and weak convergence, respectively. For every $x$, $y$, $z$, and $w$ in $\mathcal{H}$, the following equation holds:
\begin{equation}\label{fp}
2\langle x - y, z - w \rangle = \|x - w\|^2 + \|y - z\|^2 - \|x - z\|^2 - \|y - w\|^2.
\end{equation}
We call the following bound Young’s inequality:
$$\langle a, b \rangle \leq \frac{\varepsilon}{2} \|a\|^2 + \frac{1}{2\varepsilon} \|b\|^2,~ \forall a, b \in \mathcal{H},~ \forall \varepsilon > 0.$$
Let $A : \mathcal{H} \to 2^{\mathcal{H}}$ be a set-valued operator. The domain, range, graph, and zero point set of $A$ are denoted by dom$A = \{x \in \mathcal{H} \mid Ax \neq \emptyset \}$, ran$A = \bigcup_{x \in \mathrm{dom} A} Ax$, gra$A = \{(x, y) \in \mathcal{H} \times \mathcal{H} \mid y \in Ax, x \in \mathrm{dom} A\}$, and zer$A = \{x \in \mathcal{H} \mid 0 \in Ax\}$, respectively. The inverse of $A$ is $A^{-1}: y \mapsto \{x \mid y \in Ax\}$. ``Id'' represents the identity operator. The operator $A$ is monotone if and only if $\langle y - w, x - z \rangle \geq 0$ for all $(x, y), (z, w) \in \mathrm{gra} A$. The operator $A$ is uniformly monotone with modulus $\phi_A: [0,+\infty)\to [0,+\infty]$ if $\phi$ is increasing, vanishes only at 0, and $\langle y - w, x - z \rangle \geq \phi_A(\|y-x\|)$ for all $(x, y), (z, w) \in \mathrm{gra} A$. The operator $A$ is maximal monotone if it is monotone, and there exists no monotone operator $A' : \mathcal{H} \to 2^\mathcal{H}$ such that gra$A'$ properly contains gra$A$. The resolvent of the maximal monotone operator $A$ is $J_A = ({\rm Id} + A)^{-1}$, then it follows that dom$J_A = \mathcal{H}$ \cite[Theorem 21.1]{Combettes2017Convex} and $J_A$ is firmly nonexpansive (FNE) \cite[Proposition 23.8]{Combettes2017Convex}, i.e.,
$$ \langle x - y, J_Ax - J_Ay \rangle \geq \|J_Ax - J_Ay\|^2,~ \forall x, y \in \mathcal{H}. $$
Using the definition of the resolvent, we can directly obtain that
$$ y = J_Ax \Leftrightarrow x - y \in Ay. $$
The operator $B$ is $\beta$-cocoercive ($\beta > 0$) if $\beta B$ is FNE, i.e.,
$$ \langle x - y, Bx - By \rangle \geq \beta \|Bx - By\|^2,~ \forall x, y \in \mathcal{H}.$$
A mapping $C$ is demiregular at $x\in\operatorname{dom} C$ if for all $u\in Cx$ and all sequences $(x^k, u^k)\in\operatorname{gra} C$ satisfying $x^k \rightharpoonup x$ and $u^k\to u$, we have $x^k\to x$. 
We define \( \Gamma_0(\mathcal{H}) \) as the set of proper, lower semicontinuous, convex functions \( f: \mathcal{H} \to (-\infty, \infty] \). Let \( f \in \Gamma_0(\mathcal{H}) \); then the subdifferential of \( f \), denoted by \( \partial f \), is defined as
\[
\partial f : x \mapsto \{ u \in \mathcal{H} \mid f(y) \geq f(x) + \langle u, y - x \rangle, \, \forall y \in \mathcal{H} \}.
\]
The subdifferential \( \partial f \) is maximal monotone \cite[Proposition 20.25]{Combettes2017Convex}. The set of minimizers of \( f \), denoted by \( \arg\min_{x \in \mathcal{H}} f(x) \), is the zero-point set of \( \partial f \). The resolvent of \( \partial f \) is also known as the proximal operator of \( f \), denoted by
\[
\text{prox}_f : x \mapsto \arg\min_{y \in \mathcal{H}} \left\{ f(y) + \frac{1}{2} \| y - x \|^2 \right\}.
\]
In cases where \( \nabla f \) exists, the Baillon--Haddad theorem states that \( \nabla f \) is \( L \)-Lipschitz continuous if and only if \( \nabla f \) is \( \frac{1}{L} \)-cocoercive. For a nonempty convex set \( C \), let \( \text{sri}(C) \) denote the strong relative interior of \( C \). When \( C \) is a nonempty closed convex set, we denote the indicator function of \( C \) by \( I_C \in \Gamma_0(\mathcal{H}) \), and the normal cone of \( C \) by \( \mathcal{N}_C = \partial I_C \). Furthermore, using the definition of the resolvent of \( \mathcal{N}_C \), we have \( J_{\mathcal{N}_C} = P_C \), where \( P_C \) represents the projection operator onto the closed convex set \( C \).
The following lemmas are essential for designing and analyzing our algorithm.
\begin{lemma}{\rm{\cite[Lemma 2.47]{Combettes2017Convex}}}\label{l1}
Let $\{c_n\}_{n=1}^\infty$ be a sequence in real Hilbert space $\mathcal{H}$ and let $C$ be a nonempty subset of $\mathcal{H}$. Suppose that for any given $c\in C$, $\left\{\|c_n-c\|\right\}_{n=1}^\infty$ converges and that every weak sequential cluster point of $\{c_n\}_{n=1}^\infty$ belongs to $C$. Then $\{c_n\}_{n=1}^\infty$ converges weakly to a point in $C$.
\end{lemma}
 \begin{lemma}{\rm{\cite[Proposition 20.37]{Combettes2017Convex}}}\label{l0}
 Let $T$ be a maximal monotone operator from a real Hilbert space $\mathcal{J}$ to $2^\mathcal{J}$, let $(x_b,u_b)_{b\in B}$ be a bounded net in {\rm{gra}}T, and let $(x,u)\in\mathcal{J}\times\mathcal{J}$. If $x_b\to x$ and $u_b\rightharpoonup u$, then $(x,u)\in$ {\rm{gra}}T.
 \end{lemma}
 \begin{lemma}{\rm{\cite[Corollary 25.5]{Combettes2017Convex}}}\label{l00}
 Let $F$ and $G$ be maximal monotone operators from a real Hilbert space $\mathcal{J}$ to $2^\mathcal{J}$ such that ${\rm{dom}} G=\mathcal{J}$. Then $F+G$ is maximal monotone.
 \end{lemma}

For more information on monotone operators, functional analysis, and convex analysis, readers are referred to the works by \cite{Combettes2017Convex,2007Functional,1955Functional,1970ca}.
\section{Algorithm and main theoretical results}\label{s3}
Algorithm \ref{AA} presents a comprehensive formulation of our proposed splitting algorithm for the three-operator problem \eqref{cp}.
\begin{algorithm}[H]
\caption{A new three-operator splitting method for solving \eqref{cp}.}
\label{AA}
\begin{algorithmic}[1]
{\STATE {\pmb{Initialization}: $k=0$; initial point $z^0,y^0 \in\mathcal{H}$; relaxed parameter $\lambda\in(0,2)$;  stepsize $\gamma\in \left(0,\frac{2\beta\lambda(2-\lambda)}{8-3\lambda}\right)$; { {$r_0=1$}} and the error tolerance $\varepsilon\in(0,1)$.}}
{\STATE{\pmb{While}~$r_k>\varepsilon$}}
{\STATE{\pmb{Compute:}}}
\begin{align}
 x^{k+1}& = J_{\gamma C}(z^k - \gamma B y^k),\label{r1} \vspace{1ex}\\
  y^{k+1}& = J_{\gamma A}(2x^{k+1} - z^k),\label{r2} \vspace{1ex}\\
  {z}^{k+1}&={ z}^{k}+\lambda({ y}^{k+1}-{ x}^{k+1}).\label{r3}
\end{align}
\STATE{\pmb{Do}  $r_{k+1}= \frac{\|{ y}^k-{ x}^k\|^2}{\|{ x}^k\|^2+1}$ and $k \leftarrow k+1$.}
\STATE{\pmb{End while}}
\end{algorithmic}
\end{algorithm}

\begin{remark}\label{re3.1}
1. If $B = 0$ and $\lambda=1$, then Algorithm~\ref{AA} reduces to the  Douglas–Rachford splitting method.

2.  For the three-operator splitting method \eqref{0}, setting either \(A = 0\) or \(C = 0\) yields the standard forward-backward splitting algorithm. In contrast, Algorithm~\ref{AA} exhibits different limiting behavior. Specifically, when \(A = 0\) and \(\lambda = 1\), we have \(J_{\gamma A} = \mathrm{Id}\), and Algorithm~\ref{AA} reduces to
\begin{equation*}
\begin{cases}
x^{k+1} = J_{\gamma C}(x^k - \gamma B y^k), \\[1ex]
y^{k+1} = 2x^{k+1} - x^k.
\end{cases}
\end{equation*}
This is exactly the reflected forward–backward (RFB) splitting method proposed in~\cite{rfb}. On the other hand, when \(C = 0\) and \(\lambda = 1\), Algorithm~\ref{AA} simplifies to
\begin{equation*}
y^{k+1} = J_{\gamma A} \bigl[ y^k - \gamma B y^k - \gamma (B y^k - B y^{k-1}) \bigr],
\end{equation*}
which coincides with the forward–reflected–backward (FRB) splitting method introduced in~\cite{frb}. Notably, both RFB and FRB admit stepsize bounds that are smaller than those of the forward-backward method. Consequently, it is natural that the stepsize bound of our Algorithm~\ref{AA} is smaller than that of the Davis–Yin splitting method, as our algorithm encompasses both RFB and FRB as special cases and inherits their stepsize restrictions.
\end{remark}
\subsection{Convergence analysis}
Before presenting the convergence results of Algorithm~\ref{AA}, we establish the following lemma. We begin by providing an equivalent characterization of the solution to problem~\eqref{cp}.
\begin{lemma}\label{kkt}
For any $\gamma > 0$, we define the set $\mathcal{U}$ by
\[
\mathcal{U} = \left\{ (z, x) \in \mathcal{H}^2:=\mathcal{H} \times \mathcal{H} \;\middle|\; x = J_{\gamma C}(z - \gamma Bx) = J_{\gamma A}(2x - z) \right\}.
\]
The following statements are valid.

{\rm 1.} If $x \in \mathrm{zer}(A + B + C)$, then there exists some $z \in \mathcal{H}$ such that $(z, x) \in \mathcal{U}$.

{\rm 2.} If $(z, x) \in \mathcal{U}$, then $x \in \mathrm{zer}(A + B + C)$. Consequently, $\mathcal{U} \neq \emptyset \;\Leftrightarrow\; \mathrm{zer}(A + B + C) \neq \emptyset$.
\end{lemma}
\begin{proof}
{\rm Assume that $x \in \mathrm{zer}(A + B + C)$. Then there exist $c \in Cx$ and $a \in Ax$ such that
\[
\gamma(a + c + Bx) = 0.
\]
Let $z = x + \gamma c + \gamma Bx$. Then
\[
x = J_{\gamma C}(z - \gamma Bx).
\]
Moreover, $x + \gamma a = x - \gamma c - \gamma Bx = 2x - z$, which implies that
\[
x = J_{\gamma A}(2x - z).
\]
Therefore, $(z, x) \in \mathcal{U}$.

Conversely, suppose that $(z, x) \in \mathcal{U}$. Then, by the definition of the resolvent, we have
\[
\begin{cases}
z - x - \gamma Bx \in \gamma Cx, \\
x - z \in \gamma Ax.
\end{cases}
\]
Adding the two inclusions yields
\[
0 \in Ax + Bx + Cx,
\]
which implies that $x \in \mathrm{zer}(A + B + C)$. This completes the proof. $\square$}
\end{proof}

Next, to establish the convergence of Algorithm \ref{AA}, we construct a Lyapunov function defined by
\begin{equation}\label{lf}
\phi_k(z^*,x^*):= \| z^{k} - z^*\|^2 + \lambda(2-\lambda)\| y^{k} - x^k\|^2 + 2\gamma\left(\frac{\beta\lambda(8-4\lambda)}{8-3\lambda} -\gamma\right)\|By^{k-1} - Bx^*\|^2,
\end{equation}
where $( z^*,x^*)\in \mathcal{U}$. We then derive a fundamental inequality for $\phi_k(z^*,x^*)$, which is crucial for the convergence analysis.
\begin{lemma}
Let the sequence $\left\{(z^k, x^k, y^k)\right\}_{k=0}^\infty$ be generated by Algorithm \ref{AA}. Then, for any $(z^*, x^*) \in \mathcal{U}$, it holds that
\begin{align}\label{ss}
\phi_{k+1}( z^*,x^*)\leq & \phi_k( z^*,x^*) - \lambda\left(2-\lambda - \frac{8-3\lambda}{2\beta\lambda}\gamma\right) \|y^k - x^k\|^2\nonumber\\
&-2\gamma\left(\frac{\beta\lambda(8-4\lambda)}{8-3\lambda} - 2\gamma\right) \|By^{k-1} - Bx^*\|^2 \nonumber\\
&- \frac{\gamma\left(8-3\lambda\right)}{2\beta\lambda}\left(2-\lambda -\frac{8-3\lambda}{2\beta\lambda}\gamma \right) \|x^{k+1} - x^k\|^2.
\end{align}
\end{lemma}
\begin{proof}
\rm{
For any $( z^*,x^*)\in\mathcal{U}$, by using the FNE of the $J_{\gamma  A}$ and $J_{\gamma C}$, along with \eqref{r1} and \eqref{r2}, we have
\begin{equation}\label{m01}
  \langle  x^{k+1}-x^*,  z^k-\gamma {By}^k- z^*+\gamma{Bx}^*\rangle\ge \|x^{k+1}-x^*\|^2,
\end{equation}
\begin{align}
2&\langle  y^{k+1}- x^{*},  x^{k+1}- x^*\rangle-\langle y^{k+1}- x^{*}, z^k- z^*\rangle\ge \| y^{k+1}- x^*\|^2.\label{m2}
\end{align}
Summing \eqref{m01} and \eqref{m2}, we obtain that
\begin{equation}\label{m1}
\langle y^{k+1}- x^{k+1}, z^k- z^*\rangle\le -\| y^{k+1}- x^{k+1}\|^2+\gamma \langle  x^{k+1}-x^*,  {Bx}^*-{By}^k\rangle.
\end{equation}
By utilizing equation \eqref{fp} and \eqref{r3}, we can deduce that
\begin{equation}\label{lp}
\begin{aligned}
\langle y^{k+1}- x^{k+1}, z^k- z^*\rangle&=\frac{1}{\lambda}\langle z^{k+1}- z^{k}, z^k- z^*\rangle\\
&=\frac{1}{2\lambda}\left(\|z^{k+1}-z^*\|^2-\|z^k-z^*\|^2-\lambda^2\|y^{k+1}-x^{k+1}\|^2\right).
\end{aligned}
\end{equation}
Substituting \eqref{lp} into \eqref{m1}, we have
\begin{equation}\label{k}
\|z^{k+1}-z^*\|^2+\lambda(2-\lambda)\|y^{k+1}-x^{k+1}\|^2\le \|z^k-z^*\|^2+2\gamma\lambda \langle  x^{k+1}-x^*,  {Bx}^*-{By}^k\rangle.
\end{equation}
Next, we handle the final cross term on the right-hand side of inequality \eqref{k}. By applying Young's inequality, we have:
\begin{equation}\label{cb}
\begin{aligned}
\langle  x^{k+1}-x^*,  {Bx}^*-{By}^k\rangle&=\langle  x^{k+1}-y^k+y^k-x^*,  {Bx}^*-{By}^k\rangle\\
&\le \langle  x^{k+1}-y^k,{Bx}^*-{By}^k\rangle-\beta\|{Bx}^*-{By}^k\|^2\\
&\le \frac{8-3\lambda}{4\beta\lambda}\|x^{k+1}-y^k\|^2-\frac{4(2-\lambda)\beta}{8-3\lambda}\|{Bx}^*-{By}^k\|^2.
\end{aligned}
\end{equation} 
Utilizing \eqref{r1} and the fact that resolvents is FNE, we have 
\begin{equation}\label{ct}
\left\langle x^{k+1}-x^k,z^{k}-z^{k-1} +\gamma (By^{k-1}-By^k)\right \rangle\ge \|x^k-x^{k+1}\|^2.
\end{equation}
Combining \eqref{r3} and \eqref{fp}, we can obtain that
\begin{equation}\label{cc}
\begin{aligned}
\langle x^{k+1}-x^k,z^{k}-z^{k-1} \rangle&={\lambda}\langle x^{k+1}-x^k,y^{k}-x^k \rangle\\
&=\frac{\lambda}{2}\left(\|x^{k+1}-x^k\|^2+\|x^{k}-y^k\|^2-\|x^{k+1}-y^k\|^2\right).
\end{aligned}
\end{equation}
Substituting \eqref{cc} into \eqref{ct}, we have
\begin{equation}\label{cl}
\begin{aligned}
\|x^{k+1}-y^k\|^2&=\|x^{k+1}-x^k\|^2+\|x^{k}-y^k\|^2-2\langle x^{k+1}-x^k,y^{k}-x^{k} \rangle\\
&\le\left(1-\frac{2}{\lambda}\right)\|x^{k+1}-x^k\|^2+\|x^{k}-y^k\|^2+\frac{2\gamma}{\lambda} \langle x^{k+1}-x^k,By^{k-1}-By^k \rangle.
\end{aligned}
\end{equation}
By multiplying both sides of inequality~(\ref{cl}) by $\frac{8-3\lambda}{4\beta\lambda}$ and adding inequality~(\ref{cb}), we obtain
\begin{equation}\label{ck}
\begin{aligned}
\langle  x^{k+1}-x^*,  {Bx}^*-{By}^k\rangle\le& -\frac{(8-3\lambda)(2-\lambda)}{4\beta\lambda^2}\|x^{k+1}-x^k\|^2+\frac{8-3\lambda}{4\beta\lambda}\|x^{k}-y^k\|^2\\
&-\frac{4(2-\lambda)\beta}{8-3\lambda}\|{Bx}^*-{By}^k\|^2\\
&+\frac{(8-3\lambda)\gamma}{2\beta\lambda^2} \langle x^{k+1}-x^k,By^{k-1}-By^k \rangle.
\end{aligned}
\end{equation} 
It follows from \eqref{ck} and \eqref{k} that
\begin{equation}\label{oo}
\begin{aligned}
\|z^{k+1}-z^*\|^2+\lambda(2-\lambda)\|y^{k+1}-x^{k+1}\|^2\le& \|z^k-z^*\|^2-\frac{(8-3\lambda)(2-\lambda)\gamma}{2\beta\lambda}\|x^{k+1}-x^k\|^2\\
&-\frac{8(2-\lambda)\beta\lambda\gamma}{8-3\lambda}\|{Bx}^*-{By}^k\|^2\\
&+\frac{(8-3\lambda)\gamma^2}{\beta\lambda} \langle x^{k+1}-x^k,By^{k-1}-By^k \rangle\\
&+\frac{(8-3\lambda)\gamma}{2\beta}\|x^{k}-y^k\|^2.
\end{aligned}
\end{equation}
Last, we estimate the finally cross term   on the right-hand side of \eqref{oo}. Using the Young's  inequality, we have:
\begin{equation}\label{qw}
\begin{aligned}
  \langle x^{k+1}-x^k,By^{k-1}-By^k\rangle=&\langle x^{k+1}-x^k,By^{k-1}-Bx^*\rangle+\langle x^{k+1}-x^k,Bx^*-By^k\rangle\\
\le&\frac{8-3\lambda}{4\beta\lambda}\|x^{k+1}-x^k\|^2+\frac{2\beta\lambda}{8-3\lambda}\|By^k-Bx^*\|^2\\
&+\frac{2\beta\lambda}{8-3\lambda}\|By^{k-1}-Bx^*\|^2.
\end{aligned}
\end{equation}
Substituting \eqref{qw} into \eqref{oo} and the definition of $\phi_k( z^*,x^*)$,  we have
\begin{equation}\label{pl}
\begin{aligned}
\phi_{k+1}( z^*,x^*)\le& \|z^k-z^*\|^2+\frac{(8-3\lambda)\gamma}{2\beta}\|x^{k}-y^k\|^2+2\gamma^2\|By^{k-1}-Bx^*\|^2\\
&-\frac{(8-3\lambda)\gamma}{2\beta\lambda}\left(2-\lambda-\frac{(8-3\lambda)\gamma}{2\beta\lambda}\right)\|x^{k+1}-x^k\|^2.
\end{aligned}
\end{equation}
Reorganizing the terms in inequality \eqref{pl} concludes the proof.
}
\end{proof}
\begin{theorem}\label{ct1}
Let $\{(z^k, x^k, y^k)\}_{k=0}^\infty$ be the sequence generated by Algorithm~\ref{AA}. Then the following statements hold:

{\rm{1}}. $\lim\limits_{k\to\infty}\| x^k- y^k\|=0$, $\lim\limits_{k\to\infty}\| x^{k+1}- x^k\|=0$, and $\lim\limits_{k\to\infty} \|By^k - Bx^*\| = 0$ for any $x^* \in \operatorname{zer}(A+B+C)$.\vspace{0.8ex}
 
{\rm{2}}. There exist $\tilde{x} \in \operatorname{zer}(A+B+C)$ and $\tilde{z} \in \mathcal{H}$ such that $(\tilde{z}, \tilde{x}) \in \mathcal{U}$ and $x^k \rightharpoonup \tilde{x}$.

{\rm{3}}. Assume that for some $\varepsilon > 0$ and a positive integer $T$, the residuals of the first $T$ iterations satisfy $\|y^k - x^k\| > \varepsilon$ for all $k = 1, \dots, T$. Then the number of iterations required to achieve $\|y^{T+1} - x^{T+1}\| \le \varepsilon$ satisfies $T = O(\varepsilon^{-2})$.

{\rm{4}}. The sequence  $\{x^k\}_{k=0}^\infty$  converges strongly to a point in ${\rm zer}(A+B+C)$ whenever any of the following holds:

{\rm (i)}. $C$ is uniformly monotone on {\rm dom}$C$;

{\rm (ii)}. $A$ is uniformly monotone on {\rm dom}$A$;

{\rm (iii)}. $B$ is  demiregular at every point $x\in{\rm zer}(A+B+C)$. 
\end{theorem}
\begin{proof}
{\rm 1: By rearranging the terms in \eqref{ss}, we obtain
\begin{equation}\label{m6}
\begin{aligned}
&\lambda\left(2-\lambda - \frac{8-3\lambda}{2\beta\lambda}\gamma\right) \|y^k - x^k\|^2+2\gamma\left(\frac{\beta\lambda(8-4\lambda)}{8-3\lambda} - 2\gamma\right) \|By^{k-1} - Bx^*\|^2\\
&+ \frac{\gamma\left(8-3\lambda\right)}{2\beta\lambda}\left(2-\lambda -\frac{8-3\lambda}{2\beta\lambda}\gamma \right) \|x^{k+1} - x^k\|^2\\
 \le& \phi_k( z^*,x^*)-\phi_{k+1}( z^*,x^*).
\end{aligned}
\end{equation}
Since the stepsize $\gamma$ and the relaxed parameter $\lambda$ satisfy $\gamma \in \left(0,\frac{2\beta\lambda(2-\lambda)}{8-3\lambda}\right)$ and $\lambda\in(0,2)$, it follows that both $2-\lambda - \frac{8-3\lambda}{2\beta\lambda}\gamma$ and $\frac{\beta\lambda(8-4\lambda)}{8-3\lambda} - 2\gamma$ are positive. Consequently, from \eqref{m6}, we deduce that
\begin{equation}\label{suma}
\sum_{k=1}^{\infty}\|x^k - y^k\|^2 < \infty, \quad \sum_{k=1}^{\infty}\|x^{k+1} - x^k\|^2 < \infty, \quad \sum_{k=1}^{\infty}\|By^k - Bx^*\|^2 < \infty.
\end{equation}
It follows from \eqref{suma} that
\[
\lim_{k\to\infty} \|x^k - y^k\| = 0, \quad \lim_{k\to\infty} \|x^{k+1} - x^k\| = 0, \quad \lim_{k\to\infty} \|By^k - Bx^*\| = 0,
\]
where $x^*$ is any given element of $\operatorname{zer}(A+B+C)$.

2: Since the sequence $\{\phi_k( z^*,x^*)\}_{k=1}^\infty$ is positive and non-increasing, the sequence converges. By Conclusion~1 and the definition of $\phi_k( z^*,x^*)$, this implies that the sequence $\{\|z^k - z^*\|\}_{k=1}^\infty$ converges for any given $z^* \in Q_A(\mathcal{U})$, where $Q_A$ maps $(z,x) \in \mathcal{H}^2$ to $z \in \mathcal{H}$. Furthermore, it follows that the sequences $\{z^k\}_{k=1}^\infty$ and $\{By^k\}_{k=1}^\infty$ are bounded in $\mathcal{H}$. By the continuity of $J_{\gamma C}$ and $J_{\gamma A}$, the sequences $\{x^k\}_{k=1}^\infty$ and $\{y^k\}_{k=1}^\infty$ are also bounded. Let $z^\infty$, $b^\infty$, and $x^\infty$ denote weak sequential cluster points of $\{z^k\}_{k=1}^\infty$, $\{By^k\}_{k=1}^\infty$, and $\{x^k\}_{k=1}^\infty$, respectively. Without loss of generality, we assume that
\[
z^{k_j} \rightharpoonup z^\infty,~ x^{k_j} \rightharpoonup x^\infty,~ By^{k_j} \rightharpoonup b^\infty,~ \text{as}~j\to\infty.
\]
From Conclusion~1, it follows that
\begin{equation*}\label{lim2}
x^{k_j+1} \rightharpoonup x^\infty,~y^{k_j} \rightharpoonup x^\infty,~ y^{k_j+1} \rightharpoonup x^\infty,~ \text{and} \quad y^{k_j} - y^{k_j+1} \to 0,~ \text{as}~j\to\infty.
\end{equation*}
By the definitions of the resolvents in \eqref{r1} and \eqref{r2}, we derive
\begin{equation}\label{cl0}
\pmb{r}^{k_j}
  \in(\mathbf{F}+\mathbf{G})\begin{pmatrix}
 z^{k_j}-x^{k_j+1}-\gamma By^{k_j}\\
 \gamma By^{k_j}\\
  y^{k_j+1}
 \end{pmatrix},
\end{equation}
 where
\[
\pmb{r}^{k_j} = \begin{pmatrix}
x^{k_j+1} - y^{k_j+1} \\
y^{k_j} - y^{k_j+1} \\
x^{k_j+1} - y^{k_j+1}
\end{pmatrix}
\]
and the operators $\mathbf{F}, \mathbf{G}: \mathcal{H}^3 \to 2^{\mathcal{H}^3}$ are defined by 
\[
\mathbf{F} + \mathbf{G} := \begin{pmatrix}
(\gamma C)^{-1} &O&O\\
O&(\gamma B)^{-1}&O \\
O&O&\gamma A
\end{pmatrix} + 
\begin{pmatrix}
O & O & -\operatorname{Id} \\
O & O & -\operatorname{Id} \\
\operatorname{Id} & \operatorname{Id} & O
\end{pmatrix}.
\]
Since $\operatorname{dom}\mathbf{G} = \mathcal{H}^3$, it can be concluded that the sum $\mathbf{F}+\mathbf{G}$ is maximal monotone by Lemma~\ref{l00}. Passing to the limit as $j \to \infty$ in equation~\eqref{cl0}, it can be deduced from Lemma~\ref{l0} that
\begin{equation}\label{inc1}
\begin{cases}
   z^\infty-\gamma b^\infty \in x^\infty+\gamma Cx^\infty,\vspace{1ex}\\
  x^\infty\in (\gamma B)^{-1} (\gamma b^\infty), \vspace{1ex}\\
  -z^\infty+ x^\infty\in \gamma  A x^\infty.
\end{cases}
\end{equation}
The inclusion \eqref{inc1} can be rewritten equivalently as
\begin{equation*}
\begin{cases}
 J_{\gamma C}( z^\infty- \gamma b^\infty)= x^\infty,\vspace{1ex}\\
  B x^\infty= b^\infty, \vspace{1ex}\\
  J_{\gamma  A}(2 x^\infty- z^\infty)= x^\infty.
\end{cases}
\end{equation*}
Therefore, it can be concluded that $(z^\infty, x^\infty) \in \mathcal{U}$. Consequently, it follows from Lemma \ref{l1} that there exists $\tilde{z} \in Q_A(\mathcal{U})$ such that $z^k \rightharpoonup \tilde{z}$.

Finally, we prove the uniqueness of the weak cluster point of the sequence $\{x^k\}_{k=1}^\infty$. 
Let $\tilde{x}$ be an arbitrary weak cluster point of $\{x^k\}_{k=1}^\infty$. We aim to show that
\(
\tilde{x} = J_{\gamma (B+C)}(\tilde{z}),
\)
where $\tilde{z}$ is the weak limit of $\{z^k\}$. Without loss of generality, assume that $x^{k_t} \rightharpoonup \tilde{x}$ as $t\to\infty$. By Conclusion~1, we have $y^{k_t}\rightharpoonup \tilde{x}$. Repeating the argument in \eqref{cl0}--\eqref{inc1} for $(x^{k_t}, y^{k_t}, z^{k_t})$ and using $z^k \rightharpoonup \tilde{z}$, we obtain $(\tilde{z}, \tilde{x}) \in \mathcal{U}$, that is,
\[
\tilde{x} = J_{\gamma  C}(\tilde{z} - \gamma B\tilde{x}).
\]
By the definition of the resolvent, we have
\[
\tilde{x} = J_{\gamma (B+C)}(\tilde{z}).
\]
Therefore, the sequence $\{x^k\}_{k=1}^\infty$ admits a unique weak cluster point, namely $J_{\gamma (C+B)}(\tilde{z})$, which implies $x^k \rightharpoonup \tilde{x}$ as $k\to\infty$. Moreover, by Lemma~\ref{kkt}, it follows that $\tilde{x} \in \operatorname{zer}(A+B+C)$.  

3: Let \(c = \lambda\left(2 - \lambda - \frac{8-3\lambda}{2\beta\lambda}\gamma\right)\). 
Since the stepsize $\gamma$ and the relaxed parameter $\lambda$ satisfy
$\gamma \in \left(0, \frac{2\beta\lambda(2-\lambda)}{8-3\lambda}\right), \quad \lambda \in (0,2)$, it follows that $c$ and $\frac{\beta\lambda(8-4\lambda)}{8-3\lambda} - 2\gamma$ are positive. By inequality \eqref{m6}, for any $k \in \mathbb{N}$, we have
\begin{equation}\label{m66}
  c \|y^k - x^k\|^2 \le \phi_k(z^*, x^*) - \phi_{k+1}(z^*, x^*). 
\end{equation}
Suppose a positive integer $T$ satisfies $\|x^k - y^k\| > \varepsilon$ for $k = 1, 2, \dots, T$, and let $T_0$ be an upper bound of $T$. Then, when $k = T_0 + 1$, we have $\|y^k - x^k\| \le \varepsilon$. Summing both sides of inequality \eqref{m66} for $k = 1$ to $T$, we obtain
\[
c \sum_{k=1}^{T} \|y^k - x^k\|^2 \le \phi_1(z^*, x^*) - \phi_{T+1}(z^*, x^*) \le \phi_1(z^*, x^*).
\]
Since $\|x^k - y^k\| > \varepsilon$ for every $k \in \{1, 2, \dots, T\}$, the above inequality yields
\[
T< \frac{\phi_1(z^*, x^*)}{c \varepsilon^2}.
\]
Therefore, the upper bound $T_0$ of $T$ can be taken as $O(\varepsilon^{-2})$.

4: (i). Since $C$ is uniformly monotone, there exists a nonnegative, monotonically increasing function $\phi_C$ with $\phi_C(0)=0$ such that $\langle x-y, u-v\rangle \ge \phi_C(\|x-y\|)$ holds for all $(x,u), (y,v)\in\operatorname{gra} C$. Taking $(x^*,z^*)\in \mathcal{U}$, it follows from the definitions of $x^{k+1}$ and $\mathcal{U}$ that we have
\begin{equation}\label{au}
\begin{aligned}
\gamma \phi_C(\|x^{k+1}-x^*\|) &\le \langle x^{k+1}-x^*, z^k-z^*+x^*-x^{k+1}+\gamma Bx^*-\gamma By^k\rangle \\
&= \langle x^{k+1}-y^{k+1}, z^k-z^*+x^*-x^{k+1}+\gamma Bx^*-\gamma By^k\rangle \\
&\quad + \langle y^{k+1}-x^*, z^k-z^*+x^*-x^{k+1}+\gamma Bx^*-\gamma By^k\rangle \\
&= \langle x^{k+1}-y^{k+1}, z^k-z^*+x^*-x^{k+1}+\gamma Bx^*-\gamma By^k\rangle \\
&\quad + \langle y^{k+1}-x^*, z^k-2x^{k+1}+y^{k+1}-z^*+x^*\rangle \\
&\quad + \langle y^{k+1}-x^*, x^{k+1}-y^{k+1}\rangle + \gamma \langle y^{k+1}-x^*, Bx^*- By^k\rangle \\
&\le \|x^{k+1}-y^{k+1}\|\Bigl(\|z^k-z^*\|+\|x^{k+1}-x^*\| + \|y^{k+1}-x^*\|\Bigr) \\
&\quad + \gamma\|By^k-Bx^*\|\left(\|y^{k+1}-x^*\|+\|x^{k+1}-y^{k+1}\|\right)\\
&\quad + \langle y^{k+1}-x^*, z^k-2x^{k+1}+y^{k+1}-z^*+x^*\rangle,
\end{aligned}
\end{equation}
The last inequality follows from the Cauchy-Schwarz inequality and the triangle inequality. Using the monotonicity of $A$ and the definitions of $y^{k+1}$ and $\mathcal{U}$, we have
$$
\langle y^{k+1}-x^*, z^k-2x^{k+1}+y^{k+1}-z^*+x^*\rangle \le 0.
$$
Substituting this into inequality \eqref{au}, we obtain
\begin{equation}\label{au_new}
\begin{aligned}
\gamma \phi_C(\|x^{k+1}-x^*\|) 
&\le\|x^{k+1}-y^{k+1}\|\Bigl(\|z^k-z^*\|+\|x^{k+1}-x^*\| + \|y^{k+1}-x^*\|\Bigr) \\
&\quad + \gamma\|By^k-Bx^*\|\left(\|y^{k+1}-x^*\|+\|x^{k+1}-y^{k+1}\|\right) \to 0,~ k\to\infty,
\end{aligned}
\end{equation}
where the convergence follows from $\lim_{k\to\infty}\|y^k-x^k\| = \lim_{k\to\infty} \|By^k-Bx^*\| = 0$ and the boundedness of the sequence $\{(x^k,y^k,z^k)\}_{k=0}^\infty$. By the monotonicity of $\phi_C$ and the fact that $\phi_C(0)=0$, we deduce that $x^{k+1}\to x^*$ as $k\to\infty$. 

(ii). By an argument analogous to the one above, replacing $C$ by $A$ and using the definition of $y^{k+1}$, we conclude that 
\begin{equation}\label{au1}
\begin{aligned}
\gamma \phi_A(\|y^{k+1}-x^*\|) &\le \langle y^{k+1}-x^*, z^*-z^k+x^{k+1}-x^*+x^{k+1}-y^{k+1}\rangle \\
&= \langle y^{k+1}-x^{k+1}, z^*-z^k+x^{k+1}-x^*+x^{k+1}-y^{k+1}\rangle \\
&\quad + \langle x^{k+1}-x^*, z^*-z^k+x^{k+1}-x^*+x^{k+1}-y^{k+1}\rangle \\
&= \langle x^{k+1}-y^{k+1}, z^k-z^*+x^*-x^{k+1}\rangle-\|x^{k+1}-y^{k+1}\|^2 \\
&\quad +  \langle x^{k+1}-x^*, z^*-x^*-\gamma Bx^*-z^k + x^{k+1}+\gamma By^k\rangle\\
&\quad +\langle x^{k+1}-x^*,x^{k+1}-y^{k+1} +\gamma Bx^*-\gamma By^k\rangle\\
&\le \|x^{k+1}-y^{k+1}\|\Bigl(\|z^k-z^*\|+\|x^{k+1}-x^*\|\Bigr) - \|y^{k+1}-x^{k+1}\|^2 \\
&\quad +  \langle x^{k+1}-x^*, z^*-x^*-\gamma Bx^*-z^k + x^{k+1}+\gamma By^k\rangle\\
&\quad + \|x^{k+1}-x^*\|\left(\|x^{k+1}-y^{k+1}\| +\gamma \|Bx^*- By^k\|\right),
\end{aligned}
\end{equation}
where $(z^*,x^*)\in \mathcal{U}$, and the last inequality follows from the Cauchy-Schwarz inequality and the triangle inequality. Using the monotonicity of $C$ and the definitions of $x^{k+1}$ and $\mathcal{U}$, we have
$$
\langle x^{k+1}-x^*,z^*-x^*-\gamma Bx^*-z^k + x^{k+1}+\gamma By^k\rangle \le 0.
$$
Substituting this into inequality \eqref{au1}, we obtain
\begin{equation}\label{au_new1}
\begin{aligned}
\gamma \phi_A(\|y^{k+1}-x^*\|) 
&\le \|x^{k+1}-y^{k+1}\|\Bigl(\|z^k-z^*\|+\|x^{k+1}-x^*\|\Bigr) - \|y^{k+1}-x^{k+1}\|^2 \\
&\quad + \|x^{k+1}-x^*\|\left(\|x^{k+1}-y^{k+1}\| +\gamma \|Bx^*- By^k\|\right) \to 0,~ k\to\infty,
\end{aligned}
\end{equation}
where the convergence follows from $\lim_{k\to\infty}\|y^k-x^k\| = \lim_{k\to\infty} \|By^k-Bx^*\| = 0$ and the boundedness of the sequence $\{(x^k,y^k,z^k)\}_{k=0}^\infty$. By the monotonicity of $\phi_A$ and the fact that $\phi_A(0)=0$, we deduce that $x^{k+1}\to x^*$ as $k\to\infty$.

(iii). From the conclusion of Part 2, let $x$ denote the weak limit of $\{x^k\}_{k=0}^\infty$; then $x\in\operatorname{zer}(A+B+C)$. Then from the conclusion of Part 1, we have that $\{y^k\}_{k=0}^\infty$ converges weakly to $x$ and $\{By^k\}_{k=0}^\infty$ converges strongly to $Bx$. Therefore, by the demiregularity of $B$, we obtain that $\{y^k\}_{k=0}^\infty$ converges strongly to $x$, and hence $\{x^k\}_{k=0}^\infty$ converges strongly to $x$ as well. This completes the proof.
} 
\end{proof}
\begin{remark}
As can be seen from equations \eqref{cl}--\eqref{inc1}, when \(x^k \approx y^k\), both \(x^k\) and \(y^k\) can be regarded as approximate solutions to the original monotone inclusion problem \eqref{cp}. Therefore, we may set the stopping criterion of the algorithm as $\frac{\|y^k - x^k\|^2}{\|x^k\|^2 + 1} \leq \varepsilon.$
\end{remark}
\subsection{Convergence rates}
An important class of applications of splitting algorithms for monotone operators lies in convex optimization. This follows from the fact that, for any proper, lower semicontinuous, and convex function $f$, its subdifferential operator $\partial f$ is maximally monotone. 

In this section, we present the convergence rate analysis of Algorithm~\ref{AA} for solving a three-block composite convex optimization problem. Specifically, we consider the problem
\begin{equation}\label{co}
\min_{x\in\mathcal{H}} \left\{F(x):= f(x) + g(x) + h(x)\right\},
\end{equation}
where $f,g,h\in\Gamma_0(\mathcal{H})$, $h$ is Fr\'{e}chet differentiable, and its gradient $\nabla h$ is $\beta^{-1}$-Lipschitz continuous. Assuming that $0\in{\rm sri}({\rm dom}f - {\rm dom}g)$, the convex optimization problem~\eqref{co} is equivalent to the following monotone inclusion problem:
\begin{equation}\label{zp}
0\in \partial F(x)=\partial f(x) + \partial g(x) + \nabla h(x).
\end{equation}

By setting $A = \partial g$, $B = \nabla h$, and $C = \partial f$ in Algorithm~\ref{AA}, we obtain the following iteration scheme for solving~\eqref{zp}:
\begin{numcases}{ }
x^{k+1} = \arg\min_{x\in\mathcal{H}}\left\{ f(x)+\langle x-y^k,\nabla h(y^k)\rangle+\frac{1}{2\gamma}\|x-z^k\|^2\right\},\label{xf1} \\
y^{k+1} = \arg\min_{y\in\mathcal{H}}\left\{ g(y)+\frac{1}{2\gamma}\|2x^{k+1}-z^k-y\|^2\right\},\label{yf1} \\
z^{k+1} = z^{k}+\lambda\left(y^{k+1}-x^{k+1}\right).\label{zz}
\end{numcases}
As will be seen in Remark~\ref{re3.1}, Algorithm~\ref{AA} serves as a general framework that encompasses the FRB and RFB methods as special cases. Consequently, establishing its convergence rate directly yields the rates for these methods.

The following lemma  establishes key properties of the objective function for the set $\mathcal{U}$ (defined in Lemma~\ref{kkt}) under the assumptions $A = \partial g$, $B = \nabla h$, and $C = \partial f$.
\begin{lemma}
For any $(z^*, x^*) \in \mathcal{U}$ and any $x, y \in \mathcal{H}$, the following inequality is satisfied:
\begin{equation}\label{efr1}
f(x)+g(y)+h(x)-F(x^*)+\frac{1}{\gamma}\langle z^*-x^*,y-x\rangle\ge 0.
\end{equation}  
\end{lemma}
\begin{proof}
{\rm From the definition of $\mathcal{U}$, we have
\[
\begin{cases}
\frac{1}{\gamma}(z^* - x^*) - \nabla h(x^*) \in \partial f(x^*), \\[3pt]
\frac{1}{\gamma}(x^* - z^*) \in \partial g(x^*).
\end{cases}
\]
By the definition of the subdifferential of a convex function, we obtain
\begin{align}
f(x) &\ge f(x^*) + \frac{1}{\gamma} \langle z^* - x^* - \gamma \nabla h(x^*),\, x - x^* \rangle,~\forall x \in \mathcal{H}, \label{si1} \\[4pt]
g(y) &\ge g(x^*) + \frac{1}{\gamma} \langle x^* - z^*,\, y - x^* \rangle,~ \forall y \in \mathcal{H}. \label{si2}
\end{align}
Summing inequalities~\eqref{si1} and~\eqref{si2}, we obtain
\begin{equation}\label{fc}
f(x) + g(y) \ge f(x^*) + g(x^*) + \frac{1}{\gamma} \langle x^* - z^*,\, y - x \rangle + \langle \nabla h(x^*),\, x^* - x \rangle.
\end{equation}
From the convexity of $h$, it follows that
\begin{equation}\label{gc}
h(x) \ge h(x^*) + \langle \nabla h(x^*),\, x - x^* \rangle,~ \forall x \in \mathcal{H}.
\end{equation}
Substituting~\eqref{gc} into~\eqref{fc}, we get
\begin{equation*}\label{fc1}
f(x) + g(y) + h(x) \ge f(x^*) + g(x^*) + h(x^*) + \frac{1}{\gamma} \langle x - y,\, z^* - x^* \rangle.
\end{equation*}
This completes the proof.}
\end{proof}

We now analyze the evolution of the objective function in \eqref{co} along the iterates generated by Algorithm \ref{AA}.\begin{lemma}
Let $\{ (z^k, x^k, y^k) \}_{k=0}^\infty$ denote the sequence generated by \eqref{xf1}-\eqref{zz}. Then, for any $(z^*, x^*) \in \mathcal{U}$, the following inequality is satisfied:
\begin{equation}\label{efr2}
\begin{aligned}
f(x^{k+1}) + g(y^{k+1}) + h(x^{k+1})  \hspace{-.25ex}- \hspace{-.25ex} F(x^*) 
\le& \frac{1}{2\beta} \| x^{k+1} - y^{k} \|^2 
+ \frac{1}{2\gamma\lambda} \hspace{-.25ex}\big( \| z^k - z^* \|^2 \hspace{-.25ex} - \hspace{-.25ex} \| z^{k+1} - z^* \|^2 \big)\\
&+\frac{1}{\gamma}\langle  z^*-x^*,x^{k+1}-y^{k+1} \rangle.
\end{aligned}
\end{equation}
\end{lemma}
\begin{proof}
{\rm Based on the first-order optimality conditions of subproblems~\eqref{xf1} and~\eqref{yf1}, we have
\begin{equation*}
(z^k - x^{k+1}) - \gamma \nabla h(y^k) \in \gamma \partial f(x^{k+1}),
\end{equation*}
and
\begin{equation*}
2x^{k+1} - z^k - y^{k+1} \in \gamma \partial g(y^{k+1}).
\end{equation*}
For any $(z^*, x^*) \in \mathcal{U}$, by the definition of the subdifferential operator for a convex function, it follows that
\begin{equation}\label{fv1}
\begin{aligned}
f(x^*) \ge& f(x^{k+1}) + \frac{1}{\gamma} \langle x^* - x^{k+1},\, z^k - x^{k+1} \rangle + \langle x^{k+1} - x^*,\, \nabla h(y^k) \rangle,
\end{aligned}
\end{equation}
and
\begin{equation}\label{gv1}
\begin{aligned}
g(x^*) \ge& g(y^{k+1}) + \frac{1}{\gamma} \langle x^* - y^{k+1},\, x^{k+1} - z^k \rangle + \frac{1}{\gamma} \langle x^* - y^{k+1},\, x^{k+1} - y^{k+1} \rangle.
\end{aligned}
\end{equation}
By summing~\eqref{fv1} and~\eqref{gv1}, we obtain
\begin{equation}\label{t1}
\begin{aligned}
f(x^*) + g(x^*) \ge& f(x^{k+1}) + g(y^{k+1}) + \frac{1}{\gamma} \langle y^{k+1} - x^{k+1},\, z^k - x^{k+1} \rangle + \langle x^{k+1} - x^*,\, \nabla h(y^k) \rangle \\
&+ \frac{1}{\gamma} \langle x^* - y^{k+1},\, x^{k+1} - y^{k+1} \rangle \\
=& f(x^{k+1}) + g(y^{k+1}) + \frac{1}{\gamma} \| x^{k+1} - y^{k+1} \|^2 + \langle x^{k+1} - x^*,\, \nabla h(y^k) \rangle \\
&+ \frac{1}{\gamma} \langle x^{k+1} - y^{k+1},\, x^* - z^k \rangle.
\end{aligned}
\end{equation}
Next, we estimate the remaining cross terms in~\eqref{t1}. Using~\eqref{fp} and~\eqref{zz}, we have
\begin{equation}\label{t2}
\begin{aligned}
\langle x^* - z^k,\, x^{k+1} - y^{k+1} \rangle
=& \langle x^* - z^*,\, x^{k+1} - y^{k+1} \rangle + \langle z^* - z^k,\, x^{k+1} - y^{k+1} \rangle \\
=& \langle x^* - z^*,\, x^{k+1} - y^{k+1} \rangle +\frac{1}{\lambda} \langle z^* - z^k,\, z^k - z^{k+1} \rangle \\
=& \langle x^* - z^*,\, x^{k+1} - y^{k+1} \rangle + \frac{1}{2\lambda} \Big( \| z^{k+1} - z^* \|^2 - \| z^k - z^* \|^2\\
& -\lambda^2 \| x^{k+1} - y^{k+1} \|^2 \Big).
\end{aligned}
\end{equation}
By the convexity of $h$ and the descent lemma, we have
\begin{equation}\label{d1}
\begin{aligned}
\langle x^{k+1} - x^*,\, \nabla h(y^k) \rangle
=& \langle x^{k+1} - y^k,\, \nabla h(y^k) \rangle + \langle y^k - x^*,\, \nabla h(y^k) \rangle \\
\ge& h(x^{k+1}) - h(y^k) - \frac{1}{2\beta} \| y^k - x^{k+1} \|^2 + h(y^k) - h(x^*) \\
=& h(x^{k+1}) - h(x^*) - \frac{1}{2\beta} \| y^k - x^{k+1} \|^2.
\end{aligned}
\end{equation}
Substituting~\eqref{t2} and~\eqref{d1} into~\eqref{t1} and simplifying, we obtain
\begin{equation}\label{t3}
\begin{aligned}
f(x^*) + g(x^*) + h(x^*) \ge& f(x^{k+1}) + g(y^{k+1}) + h(x^{k+1})
+ \frac{1}{\gamma} \langle x^* - z^*,\, x^{k+1} - y^{k+1} \rangle \\
&+ \frac{1}{2\gamma\lambda} \left( \| z^{k+1} - z^* \|^2 - \| z^k - z^* \|^2\right) + \frac{2-\lambda}{2\gamma}\| x^{k+1} - y^{k+1} \|^2\\
&- \frac{1}{2\beta} \| x^{k+1} - y^k \|^2.
\end{aligned}
\end{equation}
Rearranging the terms in inequality~\eqref{t3} completes the proof.

}
\end{proof}

Finally, building on the previously established descent properties of Algorithm \ref{AA} and the structural characteristics of the problem, we derive the convergence rate of Algorithm \ref{AA} .
\begin{theorem}
Let $\left\{ ( z^k, x^k, y^k) \right\}_{k=0}^\infty$ be the sequence generated by \eqref{xf1}-\eqref{zz}. For every $K \in \mathbb{N}$, define
\[
x_{\mathrm{av}}^K = \frac{1}{K} \sum_{k=0}^{K-1}x^{k+1},~ y_{\mathrm{av}}^K = \frac{1}{K} \sum_{k=0}^{K-1}y^{k+1}.
\]
Then the following hold:
 $\left\|x_{\mathrm{av}}^K - {y}_{\mathrm{av}}^K\right\| = O\left(\frac{1}{K}\right)$ and  $f({x}_{\mathrm{av}}^K) + g(y_{\mathrm{av}}^K)  + h(x_{\mathrm{av}}^K) - F(x^*)= O\left(\frac{1}{K}\right)$.
\end{theorem}
\begin{proof}
{\rm For any \( K \in \mathbb{N} \) with \( K \geq 1 \), define the averages
\[
x_{\mathrm{av}}^K := \frac{1}{K} \sum_{k=0}^{K-1} x^{k+1}, \quad
y_{\mathrm{av}}^K := \frac{1}{K} \sum_{k=0}^{K-1} y^{k+1}.
\]
Then, it holds that
\begin{equation*}\label{ere1}
\begin{aligned}
\left\| x_{\mathrm{av}}^K - y_{\mathrm{av}}^K \right\|
&= \frac{1}{K} \left\| \sum_{k=0}^{K-1} (x^{k+1} - y^{k+1}) \right\| 
= \frac{1}{\lambda K} \left\| \sum_{k=0}^{K-1} (z^k - z^{k+1}) \right\| 
= \frac{1}{\lambda K}  \left\| z^K - z^0 \right\|.
\end{aligned}
\end{equation*}
From Theorem~\ref{ct1}, we know that the sequence \( \{ z^k \}_{k=0}^\infty \) is bounded. Thus,
\begin{equation*}\label{cr}
\left\| x_{\mathrm{av}}^K - y_{\mathrm{av}}^K \right\|
\leq \frac{\sup_{K \in \mathbb{N}} \{ \| z^K - z^0 \| \}}{\lambda K}.
\end{equation*}
Summing both sides of \eqref{efr2} from \( k = 0 \) to \( K - 1 \) yields
\begin{equation*}\label{ef}
\begin{aligned}
\sum_{k=0}^{K-1} \left[ f(x^{k+1}) + g(y^{k+1}) + h(x^{k+1}) \right]
\leq &~ K F(x^*) + \frac{1}{2\gamma\lambda} \| z^0 - z^* \|^2 + \frac{1}{2\beta} \sum_{k=0}^{K-1} \| x^{k+1} - y^k \|^2 \\
&+ \frac{1}{\gamma} \sum_{k=0}^{K-1} \langle z^* - x^*, x^{k+1} - y^{k+1} \rangle.
\end{aligned}
\end{equation*}
From Theorem~\ref{ct1}, we have 
\(\sum_{k=1}^{\infty} \| x^{k} - y^{k} \|^2 < \infty\) and 
\(\sum_{k=1}^{\infty} \| x^{k+1} - x^{k} \|^2 < \infty\),  
which implies 
\(\sum_{k=1}^{\infty} \| x^{k+1} - y^{k} \|^2 < \infty\).  
Applying Jensen’s inequality yields
\begin{equation}\label{ef1}
\begin{aligned}
f(x_{\mathrm{av}}^K) + g(y_{\mathrm{av}}^K) + h(x_{\mathrm{av}}^K) - F(x^*)
\leq &~ \frac{1}{2\gamma\lambda K} \| z^0 - z^* \|^2 
+ \frac{1}{2\beta K} \sum_{k=0}^{\infty} \| x^{k+1} - y^k \|^2 \\
&+ \frac{1}{\gamma} \langle z^* - x^*, x_{\mathrm{av}}^K - y_{\mathrm{av}}^K \rangle.
\end{aligned}
\end{equation}
By the Cauchy–Schwarz inequality, it follows that
\begin{equation}\label{ef2}
\begin{aligned}
\langle z^* - x^*, x_{\mathrm{av}}^K - y_{\mathrm{av}}^K \rangle
&\leq \| x^* - z^* \| \, \| x_{\mathrm{av}}^K - y_{\mathrm{av}}^K \| \\
&\leq \frac{\| x^* - z^* \| \, \sup_{K \in \mathbb{N}} \{ \| z^K - z^0 \| \}}{\lambda K}.
\end{aligned}
\end{equation}
Substituting  \eqref{ef2} into \eqref{ef1} gives
\begin{equation}\label{ef7}
\begin{aligned}
f(x_{\mathrm{av}}^K) + g(y_{\mathrm{av}}^K) + h(x_{\mathrm{av}}^K) - F(x^*)
\leq &~ \frac{1}{2\gamma\lambda K} \| z^0 - z^* \|^2
+ \frac{1}{2\beta K} \sum_{k=0}^{\infty} \| x^{k+1} - y^k \|^2 \\
&+ \frac{\| x^* - z^* \| \, \sup_{K \in \mathbb{N}} \{ \| z^K - z^0 \| \}}{\gamma\lambda K}.
\end{aligned}
\end{equation}
Setting \( x = x_{\mathrm{av}}^K \) and \( y = y_{\mathrm{av}}^K \) in \eqref{efr1}, and applying the Cauchy–Schwarz inequality, we obtain
\begin{equation}\label{ef10}
\begin{aligned}
f(x_{\mathrm{av}}^K) + g(y_{\mathrm{av}}^K) + h(x_{\mathrm{av}}^K) - F(x^*)
\geq &~ -\frac{1}{\gamma} \| x^* - z^* \| \, \| x_{\mathrm{av}}^K - y_{\mathrm{av}}^K \| \\
\geq &~ -\frac{\| x^* - z^* \| \, \sup_{K \in \mathbb{N}} \{ \| z^K - z^0 \| \}}{\gamma\lambda K}.
\end{aligned}
\end{equation}
Therefore, by combining \eqref{ef7} and \eqref{ef10}, we conclude that
\[
f(x_{\mathrm{av}}^K) + g(y_{\mathrm{av}}^K) + h(x_{\mathrm{av}}^K) - F(x^*) = O\!\left( \frac{1}{K} \right).
\]
This completes the proof.}
\end{proof}
\begin{remark}
  As Algorithm~\ref{AA} encompasses the RFB and FRB methods as special cases, the derived convergence rate results consequently establish the corresponding rates for these algorithms when applied to convex optimization problems.
\end{remark}
\section{Numerical tests}\label{s4}
This section evaluates Algorithm~\ref{AA} by applying it to the portfolio optimization problem and the soft-margin support vector machine (SVM) problem. The primary goal of these numerical experiments is to validate the theoretical properties established in the preceding sections,  particularly the convergence behavior of the proposed algorithm.  To ensure fairness in the experiments, the parameters of each algorithm are set as large as possible while ensuring convergence.  All numerical experiments were implemented in MATLAB R2024a on a desktop computer equipped with an Intel Core i5-10210U processor running at 1.60 GHz and 8 GB of RAM.
\subsection{Portfolio optimization}
In this subsection, we evaluate the performance of Algorithm~\ref{AA} on the classical Markowitz portfolio optimization problem:
\begin{equation}\label{po}
\min_{x\in\mathbb{R}^n} h(x) = \frac{1}{2}x^\top Qx \quad \text{s.t.} \quad m^\top x \geq r,\quad \sum_{i=1}^n x_i = 1,\quad x_i \geq 0,
\end{equation}
where $Q\succeq 0$ is the covariance matrix, $m\in\mathbb{R}^n_+$ is the vector of expected returns, and $r>0$ is the target return. Problem~\eqref{po} can be reformulated as the sum of three functions:
\begin{equation}\label{mvo1}
\min_{x\in\mathbb{R}^n} ~h(x) + I_S(x) + I_D(x),
\end{equation}
where $S = \{x\in\mathbb{R}^n \mid x \ge 0,\; \sum_{i=1}^n x_i = 1\}$ denotes the simplex, and $D = \{x\in\mathbb{R}^n \mid m^\top x \ge r\}$ is a halfspace. In this formulation, we set $A = \partial I_S$, $B = \nabla h$, and $C = \partial I_D$. The gradient $\nabla h$ is Lipschitz continuous with constant $L = \|Q\|$, and the resolvent operators correspond to projections onto the respective sets: $J_{\gamma A} = P_S$ and $J_{\gamma C} = P_D$.

To simulate, we follow exactly the same data generation procedure and initial point settings as in~\cite{johnstone2021single}. The dimension is set to $n = 10000$. The matrix $Q$ is constructed as $Q = \frac{1}{n} Q_0 Q_0^\top$, where the entries of $Q_0$ are independently drawn from the standard normal distribution. The return vector $m$ has entries independently drawn from the uniform distribution on $[0, 100]$. The target return is set to $r = \delta_r \cdot \frac{1}{n}\sum_{i=1}^n m_i$ with $\delta_r \in \{0.5, 0.8, 1.0, 1.5\}$. The initial point is $x^0 = [1/n, \ldots, 1/n]^\top$.

We compare Algorithm~\ref{AA} against the Davis--Yin three-operator splitting (DYS) method~\eqref{0} and the backward-forward-reflected-backward (BFRB) splitting method~\cite{rieger2020backward}, as all three algorithms have comparable computational complexity per iteration. To ensure a fair comparison, the parameters of each algorithm are set as large as possible while still guaranteeing convergence. The stopping criterion is based on the relative residual:
\begin{equation}
\frac{\|y^k - x^k\|^2}{\|x^k\|^2 + 1} \leq 10^{-10}.
\end{equation}
To comprehensively evaluate the performance of the three algorithms, we consider the following merit function inspired by~\cite{johnstone2021single}, which simultaneously captures the objective value and constraint violations:
\begin{equation}
F(x) = h(x) - \min(0, m^\top x - r) + \left|\sum_{i=1}^n x_i - 1\right| - \min\left(0, \min_{i\in\{1,\dots,n\}} x_i\right).
\end{equation}

We present the numerical results for the case $\delta_r = 1.5$ in Figure~\ref{fig:portfolio}. The left panel shows the relative residual versus iteration count, while the right panel plots the merit function value $F(x^k)$ over iterations. 
\begin{figure}[tb]
\centering
\begin{minipage}{0.4\linewidth}
  \includegraphics[scale=0.45]{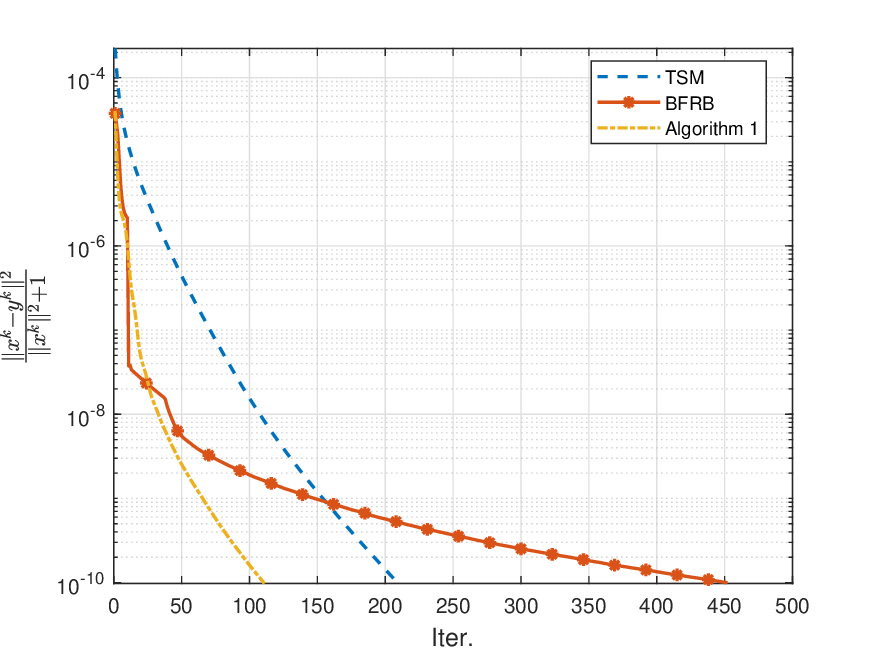}
\end{minipage}
\begin{minipage}{0.4\linewidth}
  \includegraphics[scale=0.45]{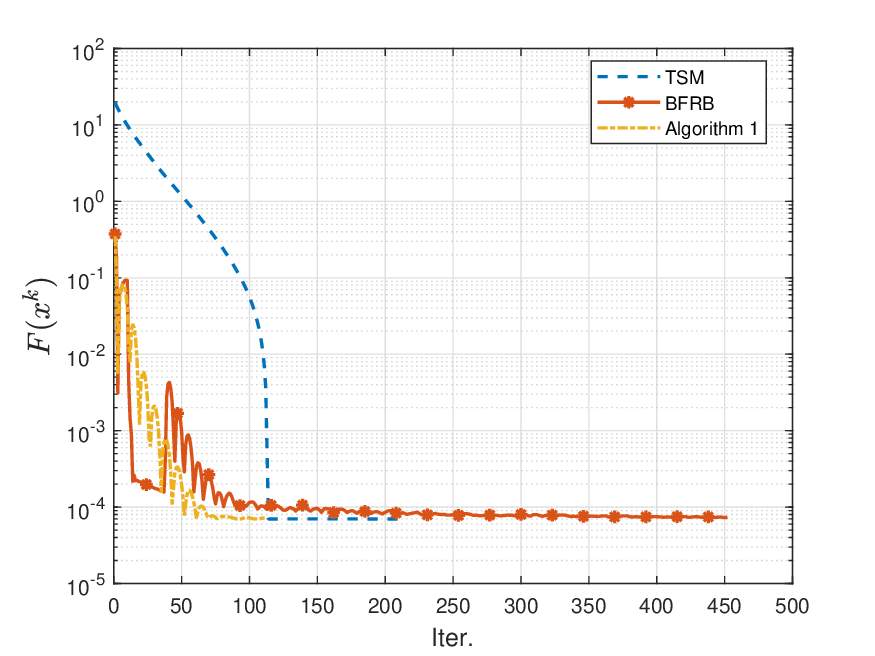}
\end{minipage}
\caption{Comparison of the numerical performance of the Davis--Yin three-operator splitting method, the backward-forward-reflected-backward (BFRB) splitting method, and Algorithm~\ref{AA} when solving problem~\eqref{po} with $\delta_r = 1.5$.}
\label{fig:portfolio}
\end{figure}

As can be observed from Figure~\ref{fig:portfolio}, both algorithms effectively solve problem~\eqref{po}. Although they have comparable per-iteration computational costs, Algorithm~\ref{AA} requires significantly fewer iterations to reach the stopping criterion. Upon termination, the final merit function values achieved by both algorithms are nearly identical, indicating that both methods yield solutions of comparable quality. Detailed numerical results for all values of $\delta_r$ are summarized in Table~\ref{tab:portfolio}, including the number of iterations, running time (in seconds), and the final merit function value. This empirical evidence supports the theoretical convergence properties established in the preceding sections.
\begin{table}[h]
\caption{Numerical results for solving problem~\eqref{po}.}\label{tab:portfolio}
\begin{center}
\begin{tabular}{cccccccccc}
\hline
\multirow{2}{*}{$\delta_r$} & \multicolumn{3}{c}{TSM\eqref{0}} & \multicolumn{3}{c}{BFRB\eqref{0}} & \multicolumn{3}{c}{Algorithm~\ref{AA}} \\
\cmidrule(l{0.2cm}r{0.2cm}){2-4} \cmidrule(l{0.2cm}r{0cm}){5-7}\cmidrule(l{0.2cm}r{0cm}){8-10}
& Time & Iter. & $F(x^k)$ & Time & Iter. & $F(x^k)$& Time & Iter. & $F(x^k)$ \\
\hline
$0.5$ & 12.19 & 353 & $2.31\times10^{-5}$ & 15.02 & 439 & ${2.77\times10^{-5}}$ & \textbf{6.16} & \textbf{180} & $\pmb{2.21\times10^{-5}}$ \\
$0.8$ & 12.28 & 343 & $\pmb{2.15\times10^{-5}}$ &16.20 &  452  & ${2.71\times10^{-5}}$ & \textbf{6.19} & \textbf{172} & ${2.32\times10^{-5}}$ \\
$1.0$ & 12.01 & 358 & ${2.89\times10^{-5}}$ &  {15.71} & {469} & $2.62\times10^{-5}$ & \textbf{6.21} & \textbf{180} & $\pmb{2.59\times10^{-5}}$ \\
$1.5$ & 11.08 & 215 & ${8.19\times10^{-5}}$ &  {23.85} & {471} & $8.59\times10^{-5}$ & \textbf{6.25} & \textbf{119} & $\pmb{8.16\times10^{-5}}$ \\
\hline
\end{tabular}
\end{center}
\end{table}

\subsection{Soft-margin support vector machine problem}
We consider the soft-margin support vector machine (SVM) problem, which can be formulated as
\begin{equation}\label{svm0}
\min_{w \in \mathbb{R}^n,\, b \in \mathbb{R}} 
\frac{1}{2c}\|w\|^2 + \sum_{i=1}^{m} \max\{1 - y_i(x_i^\top w + b),\, 0\},
\end{equation}
where $\{x_i\}_{i=1}^m \subset \mathbb{R}^n$ denote the input vectors and $\{y_i\}_{i=1}^m \subset \{-1,1\}$ are the corresponding output labels. 
Support vector machines were originally proposed by Cortes and Vapnik~\cite{1995Support} and have since been extensively used in machine learning, statistics, and pattern recognition. 
In problem~\eqref{svm0}, the term $\tfrac{1}{\|w\|}$ represents the margin between the classifier $w^\top x + b = 1$ and the feature space. Generally, the larger  $\frac{1}{\|w\|}$ is, the more robust the classifier becomes. By defining $a_i = (x_i^\top,1)^\top$, $v = (w^\top,b)^\top$, and $f(v) = \frac{1}{2c}\|w\|^2$, problem~\eqref{svm0} can be equivalently rewritten as
\begin{equation}\label{svm}
\min_{v \in \mathbb{R}^{n+1}} f(v) + \sum_{i=1}^{m} g_i(v),
\end{equation}
where $g_i(v) = \max\{1 - y_i a_i^\top v,\, 0\}$. It is straightforward to verify that $\nabla f$ is Lipschitz continuous with constant $\frac{1}{c}$.

Following the space reformulation technique proposed by Brice\~{n}o--Arias~\cite{2015Forward}, problem~\eqref{svm} can be reformulated as the following three-operator inclusion problem:
\begin{equation}\label{up1}
\begin{aligned}
0 &\in 
\left(A_1 \times A_2 \times \cdots \times A_m\right)(\boldsymbol{v}) 
+ \left(\frac{1}{m}Q \times \cdots \times \frac{1}{m}Q\right)(\boldsymbol{v}) 
+ \mathcal{N}_{V_m}(\boldsymbol{v}) \\
&:= \boldsymbol{A}(\boldsymbol{v}) + \boldsymbol{B}(\boldsymbol{v}) + \boldsymbol{C}(\boldsymbol{v}),
\end{aligned}
\end{equation}
where $A_i = \partial g_i$, $Q = \nabla f$, $V_m = \{(v_1, v_2, \ldots, v_m) \mid v_1 = v_2 = \cdots = v_m \in \mathbb{R}^{n+1}\}$, and $\mathcal{N}_{V_m}$ denotes the normal cone operator of the subspace $V_m$. We apply both the Davis–Yin three-operator splitting method and Algorithm~\ref{AA} to solve problem~\eqref{up1} and compare their numerical performance. 
In this setting, it can be verified that $\boldsymbol{B}$ is $mc$-cocoercive, and the Davis–Yin method coincides with the generalized forward–backward method~\cite{raguet2019note,2011Generalized}. 
The proximal operator of $g_i$ admits the following closed-form expression:
\[
{\rm prox}_{\gamma g_i}(u)
= u - \frac{y_i}{\|a_i\|^2} \max\big\{ \min\{y_i a_i^\top u - 1,\, 0\},\, -\|a_i\|^2 \gamma \big\} a_i.
\]

In the experiments, the parameters of each algorithm are set as large as possible while ensuring convergence. We test two widely used benchmark datasets, namely \textit{heart} and \textit{COD-RNA}, obtained from the LIBSVM repository~\cite{CC01a}. 
The hyperparameter in problem~\eqref{svm} is set to $c = 20$. 
For each dataset, $80\%$ of the samples are used for training, while the remaining $20\%$ are reserved for testing the classifier accuracy. 
Let $\{(x_i, y_i)\}_{i=1}^{t}$ denote the testing samples. 
The testing accuracy is computed as
\[
{\rm ACC}_k := 1 - \frac{1}{2t} \sum_{i=1}^{t} \Big|{\rm sign}\big(\langle w^k, x_i \rangle + b^k\big) - y_i \Big|.
\]
Figure~\ref{fig5} illustrates the testing accuracy of the classifier across iterations for both methods.
\begin{figure}[tb]
\centering
\begin{minipage}{0.4\linewidth}
  \includegraphics[scale=0.45]{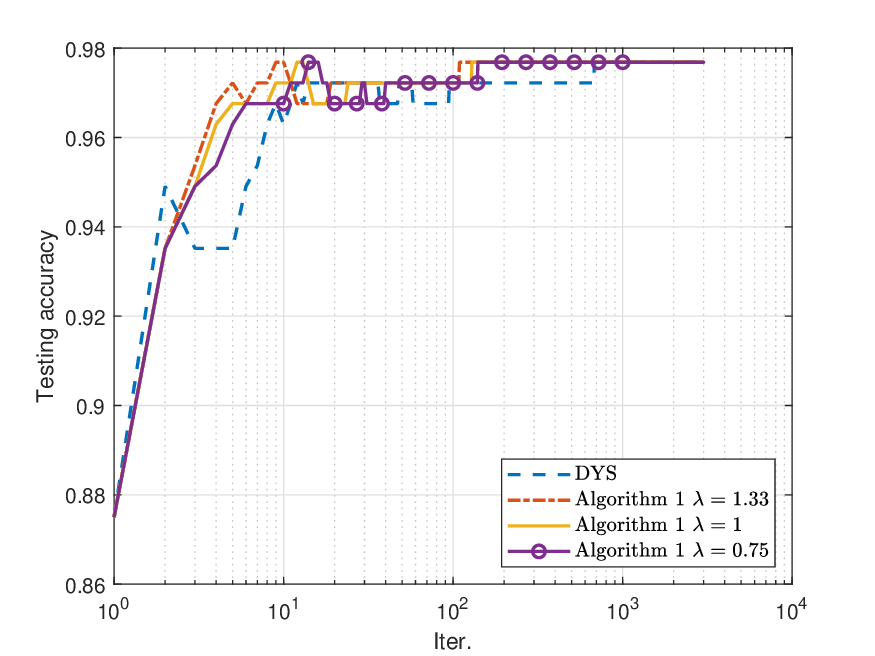}
\end{minipage}
\begin{minipage}{0.4\linewidth}
  \includegraphics[scale=0.45]{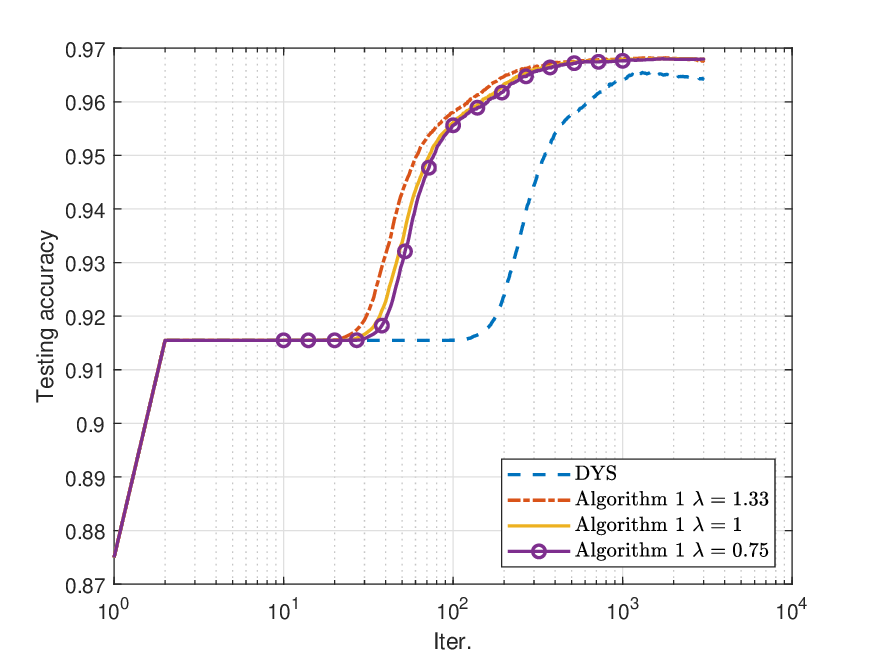}
\end{minipage}
\caption{Comparison of the numerical performance of the Davis–Yin three-operator splitting method and Algorithm~\ref{AA} for solving problem~(\ref{svm}) on the two data sets: \textit{heart} (left) and \textit{COD-RNA} (right).}
\label{fig5}
\end{figure}
Based on Figure~\ref{fig5}, both the Davis–Yin three-operator splitting method (DYS) and Algorithm~\ref{AA} can effectively solve the SVM problem and achieve comparable classification accuracy. 
For the smaller-scale dataset \textit{heart} ($m=216$), DYS exhibits higher computational efficiency, whereas for the larger-scale dataset \textit{COD-RNA} ($m=47628$), Algorithm~\ref{AA} demonstrates superior efficiency compared with DYS. In Table \ref{tab1}, we report more detailed numerical results. Note that when the relaxation coefficient $\lambda = \tfrac{4}{3}$, the upper bound of the stepsize $\gamma$ reaches its maximum value. Moreover, when the relaxation coefficient is chosen close to $\tfrac{4}{3}$, the computational efficiency of the algorithm can be improved to some extent.
\begin{table}[h]\caption{Numerical results for solving problem (\ref{svm}).}\label{tab1}
\begin{center}
\begin{threeparttable}
\begin{tabular}{ccccccc}
\hline
\centering
{\multirow{2}{*}{{Algorithm}}}&\multicolumn{3}{c}{\textit{heart} ($m=216$)}&\multicolumn{3}{c}{\textit{COD-RNA} ($m=47628$) }\\
     \cmidrule(l{0.2cm}r{0.2cm}){2-4}
     \cmidrule(l{0.2cm}r{0cm}){5-7}
&Time(s)&Accuracy&Margin&Time(s)&Accuracy&Margin\\
\hline
DYS&  18.05&  \pmb{97.69}\%&   16.12 &   284.92  &   96.42\%&  730.70 \\
Algorithm \ref{AA} $\lambda=0.75$&  17.35&  97.69\% &   16.24 &   200.22 &  96.77\%&748.50\\
Algorithm \ref{AA} $\lambda=1$&  17.21&  97.69\%&  16.24 &   196.06 &  96.80\%&  753.37\\
Algorithm \ref{AA} $\lambda=1.33$& \pmb{17.11}&  \pmb{97.69}\%&  \pmb{16.24} &   \pmb{182.85} & \pmb{ 96.80}\%&  \pmb{755.50}\\
\hline
\end{tabular}
\end{threeparttable}
\end{center}
\end{table}
\section{Discussions and conclusions}\label{s5}
In this paper, we introduce a new three-operator splitting method for solving the zero-finding problem for the sum of two maximal monotone operators and one cocoercive operator. Unlike the Davis--Yin three-operator splitting method, the proposed method generalizes the reflected forward-backward, Douglas--Rachford, and forward-reflected-backward splitting methods. Under mild assumptions, we prove global convergence of the proposed method. Since the resolvents \(J_{\gamma A}\) and \(J_{\gamma C}\) are generally nonlinear for maximal monotone operators \(A\) and \(C\), the Lyapunov function in \eqref{lf} differs from those used for RFB and FRB. Consequently, the admissible stepsize bound for Algorithm~\ref{AA} is slightly more restrictive than those for RFB and FRB (indeed, \(\tfrac{2\beta}{5} < \tfrac{\beta}{2}\)). This observation raises a natural question: can one devise an alternative Lyapunov (or evaluation) function different from \eqref{lf} that permits larger stepsizes? Moreover, the convergence rate established in this paper is sublinear. If certain operators possess strong monotonicity properties, can the algorithm achieve a linear convergence rate?

{\bibliographystyle{siam}}
{\bibliography{sn-bibliography}}
\end{document}